\documentclass{article}
\usepackage{amsmath}
\usepackage{textcomp}
\usepackage{amsthm}
\usepackage{amsfonts}
\usepackage{amssymb}
\usepackage{color}
\usepackage{subfig}
\usepackage{graphicx}
\usepackage[colorlinks, linktocpage, citecolor = purple, linkcolor = purple]{hyperref}
\usepackage[margin = 0.8in]{geometry}
\usepackage{listings}
\usepackage{xcolor}
\usepackage{mdframed} 
\usepackage{varwidth}
\usepackage{float}
\usepackage{tikz,tikz-cd}
\usepackage{array}

\newtheorem{theorem}{Theorem}[section]
\newtheorem{lemma}[theorem]{Lemma}
\newtheorem{proposition}[theorem]{Proposition}

\theoremstyle{remark}
\newtheorem{remark}[theorem]{Remark}
\theoremstyle{definition}
\newtheorem{definition}[theorem]{Definition}

\renewcommand{\bar}[1]{\overline{#1}}

\newcommand{\fjulcom}{\mathcal{K}_{\mathbb{C}}}
\newcommand{\fjulhyp}{\mathcal{K}_{\mathbb{H}}}
\newcommand{\julcom}{\mathcal{J}_{\mathbb{C}}}
\newcommand{\julhyp}{\mathcal{J}_{\mathbb{H}}}
\newcommand{\mancom}{\mathcal{M}_{\mathbb{C}}}
\newcommand{\manhyp}{\mathcal{M}_{\mathbb{H}}}

\begin{document}
\allowdisplaybreaks 

\title{Julia and Mandelbrot sets for dynamics over the hyperbolic numbers}

\author{Vance Blankers, Tristan Rendfrey, Aaron Shukert, Patrick D. Shipman}


\maketitle

\begin{abstract}
Julia and Mandelbrot sets, which characterize bounded orbits in dynamical systems over the complex numbers, are classic examples of fractal sets. We investigate the analogs of these sets for dynamical systems over the hyperbolic numbers.  Hyperbolic numbers, which have the form $x+\tau y$ for $x,y \in \mathbb{R}$, and   $\tau^2 = 1$ but $\tau \neq \pm 1$, are the natural number system in which to encode geometric properties of the Minkowski space $\mathbb{R}^{1,1}$.  We show that  the hyperbolic analog of the Mandelbrot set parameterizes connectedness of hyperbolic Julia sets. We give a wall-and-chamber decomposition of the hyperbolic plane in terms of these Julia sets.
\end{abstract}

\section{Introduction}\label{sec:intro}

The Mandelbrot set, arising from the study of dynamical systems on the complex plane, has been an object of interest ever since its introduction by Robert W. Brooks and Peter Matelski \cite{BrooksMatelski}.  With its combination of simplicity of definition and complexity of structure, the set exhibits one of the most classical fractal patterns in mathematics.  

The Mandelbrot set gives the set of complex \textit{parameter values} $c$ for which the orbit of the initial point $z_0=0$ is bounded under iterations of the map $f_c: \mathbb{C} \rightarrow \mathbb{C}$ defined by
\begin{align*}
f_c(z) \doteq z^2 + c.
\end{align*}

\begin{definition}
The \emph{Mandelbrot set} $\mancom$ is the set of complex numbers $c\in\mathbb{C}$ for which there exists some $B_c\in\mathbb{R}$ such that for all $n\in\mathbb{N}$, the inequality $\left|f_c^n(0)\bar{f^n_c(0)}\right| < B_c$ is satisfied.
\end{definition}

The left panel of Fig.~\ref{complexJuliaset} shows the Mandelbrot set. 

Julia sets, studied by the pioneers of complex dynamics Gaston Julia and Pierre Fatou, are subsets of complex \textit{phase space} and also exhibit fractal structure. 
\begin{definition}
Fix a polynomial $f:\mathbb{C}\to\mathbb{C}$. The \emph{filled Julia set associated to $f$}, denoted by $\fjulcom(f)$, is the set of values $z_0 \in \mathbb{C}$ for which there exists some $B_{z_0}\in \mathbb{R}$ such that for all $n\in\mathbb{N}$, the inequality $\left|f^n(z_0)\bar{f^n(z_0)}\right| < B_{z_0}$ is satisfied.  The \emph{Julia set associated to $f$}, denoted by $\julcom(f)$, is the boundary of $\fjulcom(f)$.  
\end{definition}

Julia sets associated to the complex quadratic polynomial $f_c$ that defines the Mandelbrot set are shown in the center and right panels of Fig.~\ref{complexJuliaset}. These examples illustrate a surprising connection of a topological nature between Mandelbrot and Julia sets given by the dichotomy theorem.

\begin{figure}[h]
\centerline{\includegraphics[width=6.2in]{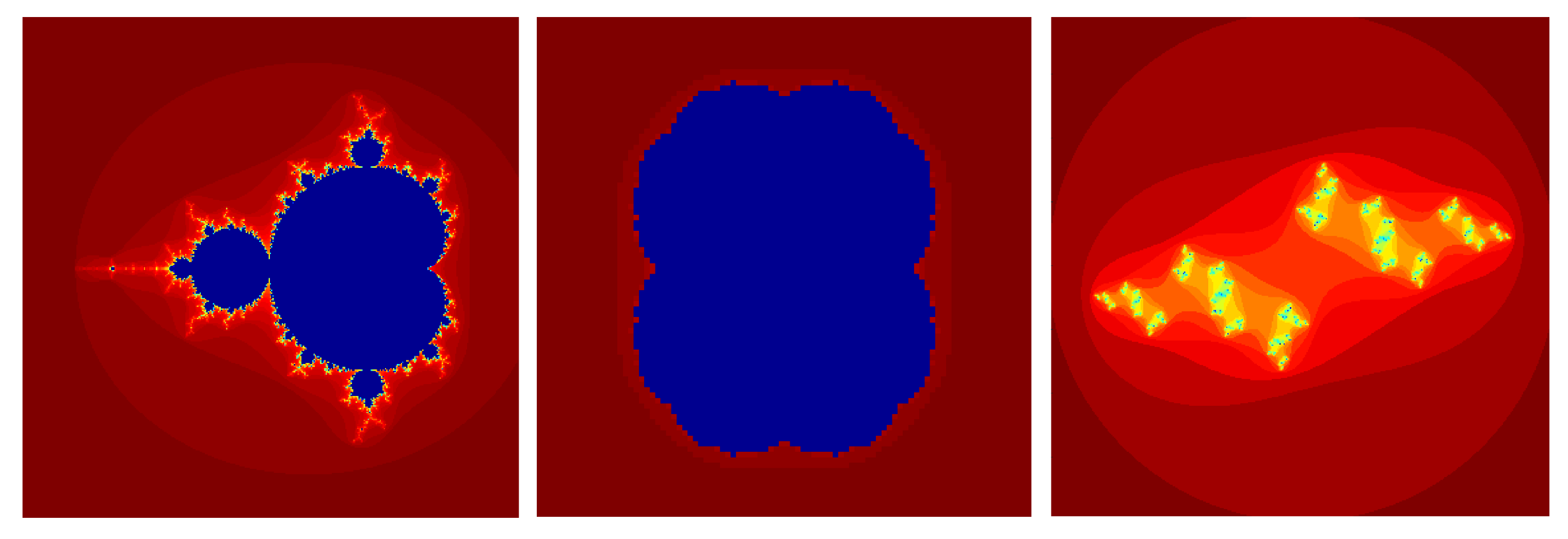}}
\caption{The Mandelbrot and examples of Julia sets.  Left panel: The Mandelbrot set is shown as a subset of parameter space $\mathbb{C}$.  Center right panels: The filled Julia set for $f_c(z) = z^2 + c$ with $c=0.2$ (center panel) and $c=-1+0.5 i$ (right panel). Colors represent the number of iterations before reaching the divergence criteria as described in \cite{devaney92}.  That is, the colors represent the iterations performed before the norm of the iterate grew larger than a chosen bound (chosen to be 4 for these simulations).  Red represents the quickest growth beyond our divergence criterion, whereas blue represents an initial condition whose orbit did not grow beyond the bound in the number of iterations we performed (200). 
}\label{complexJuliaset}
\end{figure}

\textbf{Dichotomy Theorem.}
\emph{The Mandelbrot set parameterizes connectedness of filled Julia sets:}  The filled Julia set $\fjulcom(f_c)$ is connected if $c$ is in the Mandelbrot set and totally disconnected otherwise.

For the examples of Fig.~\ref{complexJuliaset}, the choice $c=0.2$ for the center panel lies in the Mandelbrot set, and the Julia set is connected, whereas the choice $c=-1+0.5i$ for the right panel lies outside the Mandelbrot set, and the Julia set is totally disconnected.  A discussion and proof of this significant result in complex dynamics may be found in \cite{devaney92}.  The Dichotomy Theorem showcases the idea of viewing $\mathbb{C}$ as both the \emph{parameter space} and the \emph{dynamical plane} for a dynamical system. 

Given the rich results for iterations of quadratic maps on the complex plane, it is natural to wonder about the behavior of dynamics on a less well-known but also very useful sibling of the complex numbers, the hyperbolic numbers, $\mathbb{H}$. This number system has connections to diverse topics such as general relativity, differential equations, and the study of abstract algebras \cite{MMST,GMST}.

We investigate the natural analogs of the Mandelbrot set and Julia sets over $\mathbb{H}$, giving an explicit description of the former. Hyperbolic Julia sets turn out to have one of four characteristics: they may be empty, the product of intervals, the product of a Cantor set and an interval, or the product of two Cantor sets. Our main result is a wall-and-chamber decomposition of the hyperbolic plane which provides a hyperbolic-number analog to the Dichotomy Theorem:

\textbf{Quadchotomy Theorem.}
\emph{The hyperbolic Mandelbrot set parameterizes connectedness of filled hyperbolic Julia sets.}

The Quadchotomy Theorem is stated explicitly as Theorem \ref{thm:quad} 


\subsection*{Structure of the Paper}


In Section \ref{sec:hypnum}, we provide an introduction to the hyperbolic numbers, emphasizing characteristic coordinates.  Section \ref{sec:hypmand} defines the hyperbolic Mandelbrot and Julia sets and gives an explicit description of the former. The main result, the Quadchotomy Theorem, is proved in Section \ref{sec:hypjul}. 

\section{Hyperbolic Numbers}
\label{sec:hypnum}


The \emph{hyperbolic numbers} $\mathbb{H}$, sometimes called \emph{motor variables, split-complex numbers, Lorentz numbers} or a wide variety of other names, can be understood in several contexts \cite{harkin04,MMST,GMST,BPS}. Algebraically, $\mathbb{H}$ can be identified with the ring $\mathbb{R}[t]/(t^2-1)$, where we call $\tau$ the image of $t$ in the quotient. Hence they are abstractly isomorphic to $\mathbb{R}\oplus\mathbb{R}$ as a module over $\mathbb{R}$, with generators $1$ and $\tau$.  In analogy to the complex numbers, we write $z=u+\tau v$ for $u,v \in \mathbb{R}$, where   $\tau^2 = 1$ but $\tau \neq \pm 1$.  


Seen as a module over $\mathbb{R}$, hyperbolic numbers admit an automorphism which acts trivially on the component generated by $1$, called \emph{hyperbolic conjugation}. If $z=x+\tau y$, the hyperbolic conjugate is $\bar{z} = x-\tau y$. Hyperbolic conjugation shares properties with complex conjugation; $z=\bar{\bar{z}}$,  $\bar{z} + \bar{w} = \bar{z + w}$, and $\bar{z}\;\bar{w} = \bar{zw}$.


We will refer to $\mathbb{H}$ as the \emph{hyperbolic plane} in analog to the complex plane; our usage is entirely distinct from the geometric notion of the plane equipped with a hyperbolic metric, which would typically be modeled with the Poincar\'{e} disk or upper halfplane. Indeed, the hyperbolic numbers are equipped with a quadratic form, but it does not give rise to a metric or norm. Instead, if $z=x+\tau y$,
\begin{align*}
||z|| = z\bar{z} = x^2 - y^2.
\end{align*}




Representing the hyperbolic number $z=x + \tau y$ as the matrix
 \begin{equation*}
z = 
\begin{bmatrix}
x&y\\
y&x\\
\end{bmatrix}
=
\begin{bmatrix}
1&1\\
-1&1\\
\end{bmatrix}
\begin{bmatrix}
x-y&0\\
0&x+y\\
\end{bmatrix}
\begin{bmatrix}
1&1\\
-1&1\\
\end{bmatrix}^{-1}
,
\end{equation*}
addition $z_1 + z_2 = (x_1 + \tau y_1) + (x_2 + \tau y_2) = (x_1 + x_2) + j(y_1 + y_2)$ and  multiplication $z_1 z_2 = (x_1 + \tau y_1)(x_2 + \tau y_2) = (x_1 x_2 + y_1 y_2) + \tau (x_1 y_2 + x_2 y_1)$ correspond respectively to matrix addition and multiplication.  The matrix approach reveals the natural \textit{characteristic coordinates} $X = x-y$ and $Y= x+y$ with which to work with hyperbolic numbers. Representing a hyperbolic number in characteristic coordinates as 
$$z = \left( \begin{array}{cc} X & 0 \\ 0 & Y \end{array} \right),$$ the hyperbolic multiplication


\begin{equation*}
z_1z_2 =
\begin{bmatrix}
x_1&y_1\\
y_1&x_1\\
\end{bmatrix}
\begin{bmatrix}
x_2&y_2\\
y_2&x_2\\
\end{bmatrix}
=
\begin{bmatrix}
x_1x_2+y_1y_2 & x_1y_2 + x_2y_1\\
x_1y_2 + x_2y_1 & x_1x_2 + y_1y_2\\
\end{bmatrix}
\end{equation*}
is simply
\begin{equation*}
 z_1 z_2 = \left( \begin{array}{cc} X_1 & 0 \\ 0 & Y_1 \end{array} \right)\left( \begin{array}{cc} X_2 & 0 \\ 0 & Y_2 \end{array} \right) = \left( \begin{array}{cc} X_1 X_2 & 0 \\ 0 & Y_1 Y_2 \end{array}\right).
\end{equation*}
In addition, the quadratic form has a simple form in characteristic coordinates;
\begin{align}
\label{rem:charnorm}
z \bar{z} = (x + \tau y)(x - \tau y) = x^2-y^2 = (x-y)(x+y) = XY.  
\end{align}


The sets $D_+ := \{z = x+\tau x\}$ and $D_- :=\{z = x - \tau x\}$ in the hyperbolic  plane where either characteristic coordinate vanishes form the axes of the characteristic coordinate system.  Note that $D_+$ and $D_-$ are closed under addition and multiplication.



\section{The Hyperbolic Mandelbrot Set}
\label{sec:hypmand}

The simple representation of multiplication for hyperbolic numbers in characteristic coordinates gives rise to hyperbolic Mandelbrot and Julia sets that contrast significantly from the classical Mandelbrot and Julia sets of complex numbers.  Our definitions for hyperbolic Mandelbrot and Julia sets closely follow the corresponding definitions over $\mathbb{C}$. If $f:\mathbb{H}\to\mathbb{H}$ is a function, we again write $f^2(z):= f(f(z))$, $f^3(z) := f(f(f(z)))$, etc.

\begin{definition}
For each $c\in\mathbb{H}$, consider the map
\begin{align*}
f_c(z) = z^2 + c.
\end{align*}
The \emph{hyperbolic Mandelbrot set} $\manhyp$ is the set of values $c\in\mathbb{H}$ for which there exists some $B_c\in\mathbb{R}$ such that for all $n\in\mathbb{N}$, the inequality $\left|f_c^n(0)\bar{f^n_c(0)}\right|< B_c$ is satisfied.
\end{definition}

\begin{definition}
Fix a polynomial $f(z):\mathbb{H}\to\mathbb{H}$. The \emph{hyperbolic filled Julia set associated to $f$}, denoted $\fjulhyp(f)$, is the set of values $z_0 \in \mathbb{H}$ for which there exists some $B_{z_0}\in \mathbb{R}$ such that for all $n\in\mathbb{N}$, the inequality $\left|f^n(z_0)\bar{f^n(z_0)}\right|<B_{z_0}$ is satisfied.  The \emph{hyperbolic Julia set associated to $f$}, denoted by $\julhyp(f)$, is the boundary of $\fjulhyp(f)$. 
\end{definition}

The similarities in definition to the complex case lead to several of the same immediate results.  We will use the fact that, as for the complex Mandelbrot set \cite{devaney92}, $\manhyp$ is invariant under conjugation. We note as well that since both complex and hyperbolic conjugation fixes $\mathbb{R} \subset \mathbb{C},\mathbb{H}$, we must have $\manhyp \cap \mathbb{R} = \mancom \cap \mathbb{R}$.


\begin{remark}
The two definitions are in many ways similar, but the Mandelbrot set is a subset of \emph{parameter space}, whereas a Julia set is said to lie in the \emph{dynamical plane}. Theorem \ref{thm:quad} makes the connection between $\manhyp$ and $\fjulhyp(f)$ explicit for $f$ quadratic.
\end{remark}




Key to determining the hyperbolic Mandelbrot and Julia sets is the observation that in characteristic coordinates the map $f_c(z) = z^2+c_1+\tau c_2$ decouples into the real quadratic map $f_c(x) = x^2 +c$ on each coordinate.  Indeed, $f_c$ can be written as 
\begin{align*}
f_c(x,y) &= (x^2 + y^2+c_1, 2xy+c_2) \\
&= \left(\frac{1}{2}\left(X^2+Y^2\right)+c_1,\frac{1}{2}\left(Y^2-X^2\right) + c_2\right).
\end{align*}
Or, writing $f_c$ as a function characteristic coordinates,
\begin{align}
\label{eq:charlog}
f_c(X,Y) &= \left(X^2 + c_1 - c_2, Y^2 + c_1 + c_2\right) \nonumber \\
&= \left(X^2 + c_X, Y^2 + c_Y \right) \\
&= \left(f_{c_X}(X),f_{c_Y}(Y)\right), \nonumber
\end{align}
where $c_X=c_1-c_2$ and $c_Y=c_1+c_2$ are representations of the constants in characteristic coordinates.  In characteristic coordinates, the map decouples into a map on each coordinate, so that under iteration we have
\begin{align}
\label{exchar}
f_c^n(X,Y) = \left(f_{c_X}^n(X),f_{c_Y}^n(Y)\right).
\end{align}
The map $f_c(x) = x^2 +c  :  \mathbb{R}\rightarrow \mathbb{R}$ (for $x,c \in \mathbb{R}$), whose behavior is well known \cite{devaney92}, is therefore key to finding hyperbolic Julia sets.    

For  $c \leq \frac{1}{4}$, the behavior of the dynamical system $x_{n+1} = f_c(x_n)$ may be understood by a change of coordinates to the well-known logistic map.  Writing
$$\rho_+ (c) = \frac{1 + \sqrt{1-4c}}{2}, \; \; \xi = \frac{1}{2}\left( 1 - \frac{x}{\rho_+(c)} \right), \; \; r = 2 \rho_+(c), $$
the dynamical system $x_{n+1} = f_c(x_n)$ becomes the logistic dynamical system $\xi_{n+1}=g_r(\xi_n)$ for $g_r(\xi) = r (1-\xi) \xi$. 

The case $c<-2$ corresponds to $r>4$, for which all orbits of the logistic map diverge to infinity except for points in a Cantor set.  For $g_r(\xi)$, the Cantor set  is contained in $[0,\eta_-] \cup [\eta_+,1]$, where $\eta_\pm = \frac{1}{2r}(r \pm \sqrt{r^2-4r})$.   For $f_c(x)$, this translates to a Cantor set contained in $[-\rho_+(c),-\gamma(c)] \cup [\gamma(c),\rho_+(c)]$, where $4\gamma^2(c) =-4c-2-2\sqrt{1-4c}$.  Note that for $c<-2$, $\gamma(c)<0$, so the Cantor set for $f_c(x)$ is bounded away from 0.   

The case $-2 \leq c \leq \frac{1}{4}$ corresponds to $1 \leq r \leq 4$, for which orbits of the logistic map are bounded for $\xi \in [0,1]$ and diverge to infinity otherwise.  That is, for $c \in [-2, \frac{1}{4}]$, the orbit $f^n_c(x)$ is bounded if and only if $-\rho_+(c) \leq x \leq \rho_+(c)$.  In this case, the fixed points are $x=\frac{1}{2}(1 \pm \sqrt{1-4c})$; there is a fixed point equal to zero only for $c=0$.   

That  $\fjulcom(f_c)\cap \mathbb{R}$ is empty for $\frac{1}{4} < c$ may be seen as follows: For any $x \in \mathbb{R}$, the minimum value of $f_c(x) - x$ is $c-\frac{1}{4}$.  Thus, for any $x_0 \in \mathbb{R}$ and positive integer $n$, $f_c^{n+1}(x_0) \geq f_c^n(x_0) + c-\frac{1}{4}$; $f_c^n(x_0) \geq x_0 + n\left(c - \frac{1}{4} \right)$.  It follows that for $c > \frac{1}{4}$,  $f_c^n(x_0) \rightarrow \infty$ as $n \rightarrow \infty$.  

In summary, we have

\begin{lemma}
\label{lem:empty}
Let $f_c(z) = z^2+c$ for $c\in \mathbb{R}$.  Then, the intersection $\fjulcom(f_c)\cap \mathbb{R}$ is

\vspace{1mm}

i) a Cantor set not containing 0 if $c < -2$, 

\vspace{1mm}

ii) the interval $-\rho_+(c) \leq x \leq \rho_+(c)$ if $-2 < c < \frac{1}{4}$, 

\vspace{1mm}

iii) empty if  $\frac{1}{4} < c$.
\end{lemma}

The decoupling of the characteristic coordinates endows $\manhyp$ with a much simpler structure than $\mancom$, as detailed in the next theorem.

\begin{theorem}
\label{thm:mandsquare}
Let $S$ be the square given by
\begin{align*}
S :&= \left\{\max\left(x-\frac{1}{4},-x-2\right)\le y\le\min\left(x+2,-x+\frac{1}{4}\right)\right\} \\
&= \left\{(X,Y) \in \left[-2,\textstyle{\frac{1}{4}}\right]^2\right\},
\end{align*}
and let $D = D_{+}\cup D_{-}$ be the union of the diagonals in $\mathbb{H}$. Then $\manhyp = S\cup D$.
\end{theorem}

\begin{figure}[h]
\centerline{\includegraphics[width=6.2in]{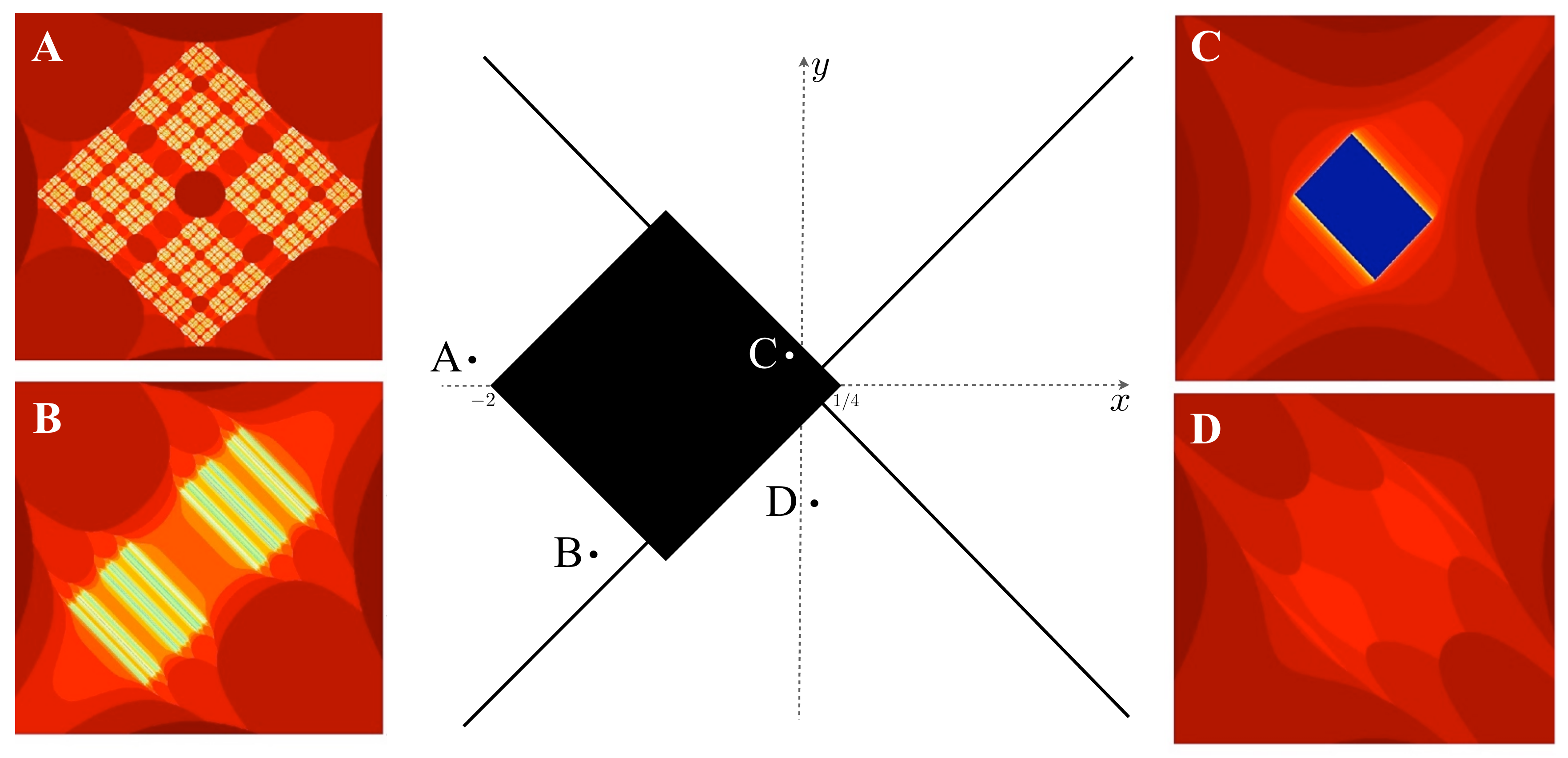}}
\caption{The hyperbolic Mandelbrot set is shown as a subset of parameter space $\mathbb{H}$ in the center panel.  The four points labeled A-D in the center panel taken as parameter values $c$ in $f_c(z) = z^2 + c$ give rise to the four types of Julia sets shown on the side panels:  (A) totally disconnected, (B) disconnected, (C) connected, and (D) empty. As with Fig.~\ref{complexJuliaset}, the colors represent the iterations performed 
before the norm of the iterate grew larger than the bound (chosen to be 4 for these simulations).  Red represents the quickest growth beyond our divergence criterion, whereas blue represents an initial condition whose orbit did not grow beyond the bound in the number of iterations we performed (200).}
\label{hyperbolicJuliaset}
\end{figure}



\begin{proof}


We need to determine the values of $c$ for which $|f_c^n(0=0+0\tau)|$ is bounded as $n$ approaches infinity. The expression
 (\ref{exchar}) for iterates of the map $f_c(z)$ in characteristic coordinates allows us to write
$|f_c^n(0)| = |f^n_{c_X}(0)f^n_{c_Y}(0)|.$
According to Lemma~\ref{lem:empty}, $f^n_{c_{X,Y}}(0)$ are bounded for (and only for) $-2 \leq c_{X,Y} \leq \frac{1}{4}$.  Thus, $|f_c^n(0)|$ is bounded for $(c_X,c_Y) \in S= \left\{(X,Y) \in \left[-2,\textstyle{\frac{1}{4}}\right]^2\right\}$.


It could also be the case that, without loss of generality, $f^n_{c_X}(0) \rightarrow 0$ but $f^n_{c_Y}(0) \rightarrow \infty$ in a manner so that their product is bounded. Since $f^n_{c_X}(0) \rightarrow 0$ only for $c_{X} = 0$, such cases occur only for $(c_X,c_Y)$ on the union $D$. $D$ is, in fact, in $\manhyp$: Since $D_+$ and $D_-$ are closed under addition and multiplication, the restrictions $f_c\big\vert_{D_\pm}:D_\pm\to D_\pm$ are well defined. But since $\left|z\bar{z}\right|= 0$ for all $z \in D$, we have that $D \subset \manhyp$.

\end{proof}

\begin{remark}
\label{rem:diaginman}
As implied by Theorem \ref{thm:quad} below, the fact that the part of $D$ outside of $S$ is in the Mandelbrot set is largely an artifact of  the fact that $D_+$ and $D_-$ are ideals of $\mathbb{H}$.
\end{remark}

\section{Hyperbolic Julia Sets}
\label{sec:hypjul}

Over the complex numbers, $\mancom$ determines the points in parameter space which correspond to connected Julia sets, and one may ask if $\manhyp$ performs the analogous role for the hyperbolic numbers. The positive answer may be given more nuance, as $\mathbb{H}$ as a parameter space admits a wall-and-chamber decomposition based on the form of $\fjulhyp (f)$, in which $\manhyp$ is the chamber corresponding to connectedness of nonempty filled Julia sets. We now develop this decomposition explicitly.


\begin{proposition}
\label{prop:juliaproduct}
For $c\in\mathbb{H}$, let $c=(c_X,c_Y)$ be its description in characteristic coordinates and let $f_c(z) = z^2+c$. Write $f_{c_X}(X) = X^2+c_X$ and $f_{c_Y}(Y) = Y^2+c_Y$. For $c_X,c_Y \neq 0$, the filled hyperbolic Julia set $\fjulhyp(f_c)$ is equal to the Cartesian product of $\fjulcom(f_{c_X})\cap\mathbb{R}$ and $\fjulcom(f_{c_Y})\cap \mathbb{R}$.
\end{proposition}
\begin{proof}
Let $z_0 = (X_0,Y_0)$ in characteristic coordinates. By equation $(1)$, $\left|f^n(z_0) \bar{f^n(z_0)}\right| = \Big| f_{c_X}^n(X_0) f_{c_Y}^n(Y_0) \Big|$ = $\Big| f_{c_X}^n(X_0)\Big| \Big| f_{c_Y}^n(Y_0) \Big|$.  According to the discussion leading to Lemma~\ref{lem:empty},  $ \lim_{n\rightarrow \infty} f_c^n(X_0) = 0$ only for $c=0$.  For $c_X,c_Y \neq 0$, 
there is a $B_{z_0}$ such that $\Big| f_{c_X}^n(X_0) f_{c_Y}^n(Y_0) \Big| < B_{z_0}$ for all $n$ if and only if there is some $M_{z_0} \in \mathbb{R}$ such that  for all $n$ $\Big| f_{c_X}^n(X_0) \Big|$, $\Big| f_{c_Y}^n(Y_0) \Big| < M_{z_o}$. But, since $c_X,X_0\in\mathbb{R}$, we have $\Big| f_{c_X}^n(X_0) \Big| < M_{z_0}$ if and only if $X_0 \in \fjulcom(f_{c_X}) \cap \mathbb{R}$, $Y_0 \in \fjulcom(f_{c_Y}) \cap \mathbb{R}$. We conclude that, for $c_X,c_Y \neq 0$,
\begin{align*}
\fjulhyp(f_c)=\left\{(X,Y): \Big| f_{c,X}^n(X) \cdot f_{c,Y}^n(Y) \Big| < B_{z_0} \right\}= \left\{ \fjulcom(f_{c_X})\cap\mathbb{R} \right\} \times \left\{ \fjulcom(f_{c_Y})\cap \mathbb{R}\right\}.
\end{align*}
\end{proof}



Examples of filled hyperbolic Julia sets are shown in the side panels of Fig.~\ref{hyperbolicJuliaset}.  The filled hyperbolic Julia may be totally disconnected (panel A), connected but not totally disconnected (panel B), connected and nonempty (panel C), or empty (panel D).  These examples represent the decomposition of $\mathbb{H}$ stated in the following analog to the Dichotomy Theorem of complex Mandelbrot and filled Julia sets and depicted in Fig.~\ref{fig:quadchotomytheorem}:

\begin{theorem}[Quadchotomy]
\label{thm:quad}

For $c\in\mathbb{H}$, let $f_c(z) = z^2+c$, and let $c = (c_X,c_Y)$ in characteristic coordinates, with $c_X,c_Y \neq 0$ . Then $\mathbb{H}$ admits a wall-and-chamber decomposition as follows:
\begin{itemize}
\item[(i)] if $c\in S$, then $\fjulhyp(f_c)$ is nonempty and connected;
\item[(ii)] if one of $c_X,c_Y$ is in $[-2,\frac{1}{4}]$ and the other is less than or equal to  $-2$, then $\fjulcom(f_c)$ is disconnected;
\item[(iii)] if $c_X,c_Y < -2$, then $\fjulcom(f_c)$ is totally disconnected;
\item[(iv)] otherwise, $\fjulhyp(f_c)$ is empty.
\end{itemize}
\end{theorem}

\begin{center}
\begin{figure}[H]
\begin{tikzpicture}[scale=2];
\draw [thick] [<->] (0,2) -- (0,0) -- (3,0);
\draw [thick] [-] (0,-3) -- (0,0) -- (-3,0);
\draw [thick](.25,.25) -- (.25,-2);
\draw [thick](.25,-2) -- (-2,-2);
\draw [thick](-2,-2) -- (-2,.25);
\draw [thick](-2,.25) -- (.25,.25);
\draw [blue][dashed] (.25,.25) -- (.25,2);
\draw [blue][dashed] (.25,.25) -- (3,.25);
\draw [blue][dashed] (-2,.25) -- (-3,.25);
\draw [blue][dashed] (-2,.25) -- (-2,2);
\draw [blue][dashed] (-2,-2) -- (-3,-2);
\draw [blue][dashed] (-2,-2) -- (-2,-3);
\draw [blue][dashed] (.25,-2) -- (.25,-3);
\draw [blue][dashed] (.25,-2) -- (3,-2);
\node [below right] at (3,0) {$X$};
\node [above left] at (0,2) {$Y$};
\node [below] at (-1,-1) {Connected};
\node [below] at (1,1) {$\varnothing$};
\node [below] at (1,-1) {$\varnothing$};
\node [below] at (1,-2.3) {$\varnothing$};
\node [below] at (-2.5,1) {$\varnothing$};
\node [below] at (-1,1) {$\varnothing$};
\node [below] at (-3,-2.5) {Totally Disconnected};
\node [below] at (-1,-2.5) {Disconnected};
\node [below] at (-3,-1) {Disconnected};
\node [below left] at (-2,-2) {$(-2,-2)$};
\node [above right] at (.25,.25) {$(\frac{1}{4},\frac{1}{4})$};
\end{tikzpicture}
\caption{The wall-and-chamber decomposition of the Quadchotomy Theorem, Theorem \ref{thm:quad}.  The hyperbolic plane in characteristic coordinates is divided into regions, as labeled, in which parameter values $(c_X,c_Y)$ yield Julia sets for $f_c(z) = z^2+c$, $c=\frac{1}{2}(c_X+c_Y)+ \tau \frac{1}{2}(c_Y-c_X)$, $c_X,c_Y \neq 0$, are connected and nonempty, disconnected but not totally disconnected, totally disconnected, or empty.}
\label{fig:quadchotomytheorem}
\end{figure}
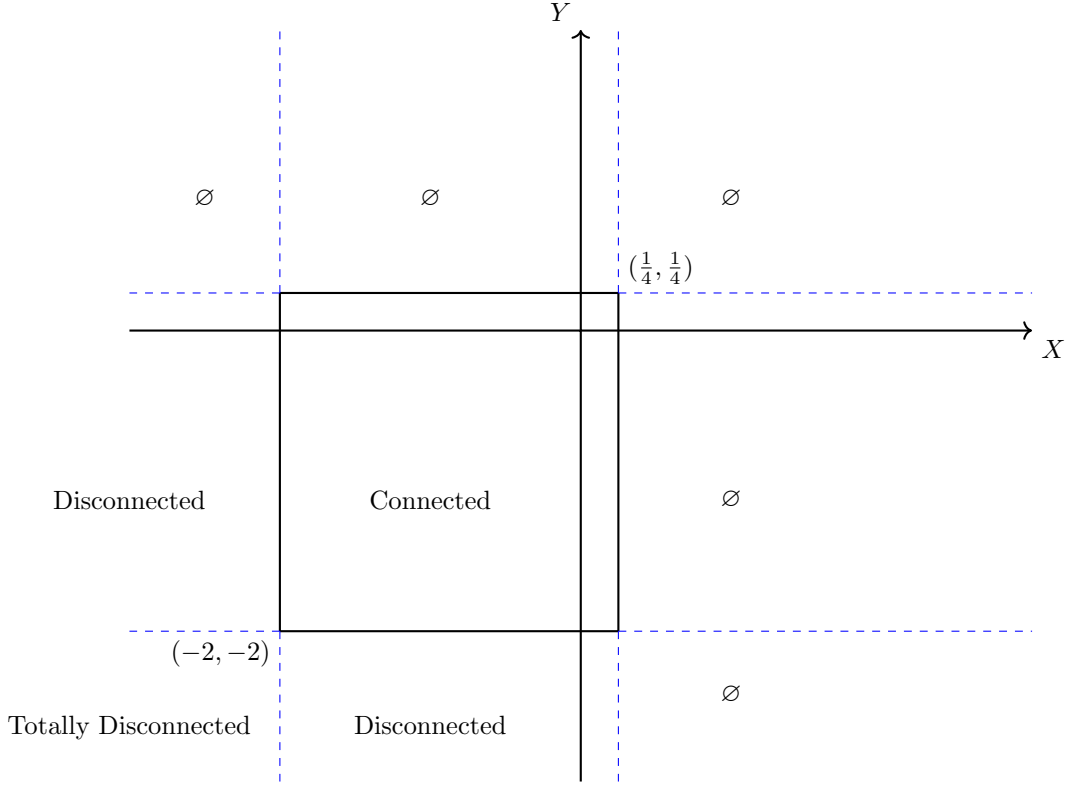

\end{center}


\begin{proof}
By Proposition \ref{prop:juliaproduct}, we need to understand $\fjulcom(f_{c_X})\cap \mathbb{R}$ and $\fjulcom(f_{c_Y})\cap \mathbb{R}$, which are given in Lemma~\ref{lem:empty}.

Part (i): When $c\in S$, $c_X,c_Y\in [-2,\frac{1}{4}]$, and $\fjulcom(f_{c_X})$ and $\fjulcom(f_{c_Y})$ are both simply connected and conjugate-invariant. Thus $\fjulcom(f_{c_X})\cap\mathbb{R}$ and $\fjulcom(f_{c_Y})\cap\mathbb{R}$ are both connected, so $\fjulhyp(f_c)$ is as well.

Part (ii): In this case exactly one of $\fjulcom(f_{c_X})\cap\mathbb{R}$ or $\fjulcom(f_{c_Y})\cap\mathbb{R}$ is connected; the other is a Cantor set. The product of a Cantor set and a connected set is a disconnected set.

Part (iii): Both $\fjulcom(f_{c_X})\cap\mathbb{R}$ and $\fjulcom(f_{c_Y})\cap\mathbb{R}$ are Cantor sets, which are totally disconnected, and the product of two totally disconnected sets is again totally disconnected.

Part (iv): By Lemma \ref{lem:empty}, at least one of $\fjulcom(f_{c_X})\cap\mathbb{R}$ or $\fjulcom(f_{c_Y})\cap\mathbb{R}$ is empty, so their product is as well.

\end{proof}

We ignored the characteristic axes in Theorem~\ref{thm:quad}.  We compute their Julia sets as follows:  If, say, $c_Y = 0$, then $\left|f_c^n(z_0) \bar{f_c^n(z_0)}\right| = \Big| f_{c_X}^n(X_0) f_{c_Y}^n(Y_0) \Big| = Y_0^{2n} \Big| f_{c_X}^n(X_0)\Big|$.  The ratio of $\left|f^{n+1}(z_0) \bar{f^{n+1}(z_0)}\right|$ to $\left|f^n(z_0) \bar{f^n(z_0)}\right|$ is 
$Y_0^{2} \frac{\Big| f_{c_X}^{n+1}(X_0)\Big|}{\Big| f_{c_X}^n(X_0)\Big|}= Y_0^{2} \frac{\Big| f_{c_X}\left(f_{c_X}^{n}(X_0)\right)\Big|}{\Big| f_{c_X}^n(X_0)\Big|} = Y_0^{2} R\left( f_{c_X}^n(X_0) \right)$, for  $R(x) = x + \frac{c_X}{x}.$
If $\Big| f_{c_X}^n(X_0)\Big|$ is unbounded, then so is this ratio since $R(x)$ approaches infinity with $x$.  That is, for $c_X>\frac{1}{4}$, the filled Julia set is empty.  For $c_X < \frac{1}{4}$, the Julia set is the product $(X_0,Y_0) \in \mathfrak{C} \times [-1,1]$ for a Cantor set $\mathfrak{C}$ contained in $\mathfrak{C} = [-\rho_+(c_X), -\gamma(c_X)] \cup [\gamma(c_X),\rho_+(c_X)]$.

\subsection*{Acknowledgements} The authors thank the Colorado State University College of Natural Sciences and Department of Mathematics for supporting the undergraduate research program in which this research was conducted in Summer 2018.



\end{document}